\documentclass[11pt]{amsart}

\usepackage{amscd,amsmath,amssymb,fancyhdr,color}
\usepackage[utf8x]{inputenc}
\usepackage{amsfonts}
\usepackage{amsthm}
\usepackage{accents}
\usepackage{graphicx}
\usepackage{float}
\usepackage{subcaption}
\usepackage{verbatim}
\usepackage{indentfirst}
\usepackage{dsfont}
\usepackage{tikz}
\usepackage{tikz-cd}
\usetikzlibrary{matrix}
\usepackage{enumerate}
\usepackage{extarrows}

\usepackage{scalerel,stackengine}
\stackMath
\newcommand\reallywidehat[1]{%
	\savestack{\tmpbox}{\stretchto{%
			\scaleto{%
				\scalerel*[\widthof{\ensuremath{#1}}]{\kern-.6pt\bigwedge\kern-.6pt}%
				{\rule[-\textheight/2]{1ex}{\textheight}}%WIDTH-LIMITED BIG WEDGE
			}{\textheight}% 
		}{0.5ex}}%
	\stackon[1pt]{#1}{\tmpbox}%
}

\usepackage{faktor}

\usepackage{BOONDOX-uprscr}

\usepackage[backref=page]{hyperref}
\renewcommand*{\backref}[1]{}
\renewcommand*{\backrefalt}[4]{%
    \ifcase #1 (Not cited.)%
    \or        (Cited on page~#2.)%
    \else      (Cited on pages~#2.)%
    \fi}

\hypersetup{
     colorlinks   = true,
     citecolor    = magenta
}

\newcommand{\K}{K\"ahler}

\DeclareMathOperator{\reg}{reg}

\DeclareMathOperator{\Proj}{Proj}
\DeclareMathOperator{\an}{an}

\numberwithin{equation}{section}

\def\eqref#1{(\ref{#1})}

\newcommand{\R}{{\mathbb R}}

\newcommand{\del}{\partial}
\newcommand{\delb}{\overline{\partial}}

\def\1{\sqrt{-1}\:}

\newcommand{\cntrct}                % contraction with a vector field
{\hspace{2pt}\raisebox{1pt}{\text{$\lrcorner$}}\hspace{2pt}}

\newcommand{\ei}{\textup{i}}

\newcommand{\codim}{\operatorname{codim}}

\newcommand{\ie}{{\em i.e. }}
\newcommand{\eg}{{\em e.g. }}

\renewcommand{\to}{\longrightarrow}

\newtheorem*{theorem*}{Theorem}

\newcounter{Mycounter}[section]
\newcounter{lemma}[section]
\newcounter{claim}[section]
\newcounter{sublemma}[section]
\newcounter{corollary}[section]
\newcounter{theorem}[section]
\newcounter{conjecture}[section]
\newcounter{proposition}[section]
\newcounter{definition}[section]
\newcounter{example}[section]
\newcounter{remark}[section]
\newcounter{problem}[section]
\newcounter{question}[section]
\makeatletter

\@addtoreset{equation}{section}

\@addtoreset{footnote}{section}

\makeatother

\usetikzlibrary{arrows,chains,matrix,positioning,scopes}

\makeatletter
\tikzset{join/.code=\tikzset{after node path={%
			\ifx\tikzchainprevious\pgfutil@empty\else(\tikzchainprevious)%
			edge[every join]#1(\tikzchaincurrent)\fi}}}
\makeatother

\tikzset{>=stealth',every on chain/.append style={join},
	every join/.style={->}}

\makeatletter
\newtheorem*{rep@theorem}{\rep@title}
\newcommand{\newreptheorem}[2]{%
	\newenvironment{rep#1}[1]{%
		\def\rep@title{\ref{##1}}%
		\begin{rep@theorem}}%
		{\end{rep@theorem}}}
\makeatother

\newreptheorem{theorem}{Theorem}
        
\begin{document}
	
\newpage

\title[Vaisman theorem for LCK spaces]{Vaisman theorem for LCK spaces}

\author{Ovidiu Preda}
\address{Ovidiu Preda \newline
\textsc{\indent Institute of Mathematics ``Simion Stoilow'' of the Romanian Academy\newline 
	\indent 21 Calea Grivitei Street, 010702, Bucharest, Romania}}
\email{ovidiu.preda18@icloud.com; ovidiu.preda@imar.ro}

\author{Miron Stanciu}
\address{Miron Stanciu \newline
\textsc{\indent University of Bucharest, Faculty of Mathematics and Computer Science\newline 
	\indent 14 Academiei Str., Bucharest, Romania\newline
	\indent and\newline
	\indent Institute of Mathematics ``Simion Stoilow'' of the Romanian Academy\newline 
	\indent 21 Calea Grivitei Street, 010702, Bucharest, Romania}}
\email{mirostnc@gmail.com; miron.stanciu@fmi.unibuc.ro}

\thanks{Ovidiu Preda was partially supported by a grant of Ministry of Research and Innovation, CNCS - UEFISCDI, project no.
	PN-III-P1-1.1-TE-2019-0262, within PNCDI III. \\
	\indent Miron Stanciu was partially supported by a grant of Ministry of Research and Innovation, CNCS - UEFISCDI, project no.
	PN-III-P4-ID-PCE-2020-0025, within PNCDI III. 
			}
\date{\today}

\begin{abstract}
Vaisman’s theorem for locally conformally K\" ahler (lcK) compact manifolds states that any lcK metric on a compact complex manifold which admits a K\" ahler metric is, in fact, globally conformally K\" ahler (gcK). In this paper, we extend this theorem to compact complex spaces with singularities.
\end{abstract}

\maketitle

\hypersetup{linkcolor=blue}
\tableofcontents

\section{Introduction}

The notion of locally conformally \K \ (lcK) manifolds was first introduced by I. Vaisman \cite{vaisman1976}. They are complex manifolds endowed with a hermitian metric satisfying the differential equation $d_\theta \omega := d\omega - \theta\wedge \omega = 0$ for some closed $1$-form $\theta$, which is called the Lee form of $\omega$. This is equivalent, as the name implies, to the $2$-form $\omega$ being locally conformal to a \K \ form. The form $\theta$ is exact if and only if $\omega$ is globally conformal to a \K \ form (gcK for short); as the point of lcK geometry is to work in non-K\" ahler geometry, we usually require this not to be the case.
Alternatively, lcK manifolds are quotients of \K \ manifolds by a discrete group of automorphisms which act as homotheties on the \K \ form. Since their introduction, they have been extensively studied (see \cite{OV} for an overview of the field and \eg \cite{mmo}, \cite{alex} for some recent results).

Grauert \cite{grauert} was the first to introduce K\" ahler forms on complex spaces, and some years later Moishezon \cite{moish} continued this study. Further results about existence of K\" ahler forms on images of K\" ahler spaces through holomorphic mappings were obtained by Varouchas in \cite{var84} and \cite{var89}.

In our previous paper \cite{PS}, we took a few steps towards transporting known results from the smooth case to lcK spaces with singularities. Mainly, we proved the above mentioned equivalent characterization of lcK manifolds \textit{via} their universal cover also holds for lcK spaces: 

\begin{theorem*}
	Let X be a complex space. Then X admits an lcK metric if and only if its universal covering $\tilde{X}$ admits a K\" ahler metric such that the deck automorphisms act on $\tilde{X}$ by positive homothethies of the K\" ahler metric.
\end{theorem*}

 We used this to prove a result about ramified coverings of lcK spaces with discrete fibers. 

\vspace{5pt}

A fundamental result in smooth lcK geometry is Vaisman's Theorem \cite{vasiman1980}, stating the dichotomy between \textit{pure lcK} and \textit{\K} on compact complex manifolds. A main ingredient in any proof of this theorem is either the $\del \delb$-Lemma or Hodge decomposition. Hence, there is no straightforward way of adapting the result to the singular setting, as neither of these fundamental theorems in complex geometry exist in the non-smooth case. Moreover, the very existence of appropriate $\del$ and $\delb$ operators is not clear (see \cite{mich} for a comparison of constructions of differential forms on normal varieties).

\vspace{3pt}

In this paper, we go further in developing the theory of lcK spaces, and show that Vaisman's theorem does hold true under reasonable assumptions when allowing the space to have singularities. This is done by carefully passing to the desingularization and applying a stronger variant of Vaisman's theorem due to \cite{APV}, which allows the lcK form $\omega$ to be degenerate along a subspace of positive codimension. 

More precisely, we prove:

\begin{reptheorem}{vaisman-singular}
		Let $(X,\omega,\theta)$ be a compact, locally irreducible, lcK space. If $X$ admits a K\"ahler metric, then $(X,\omega,\theta)$ is gcK.
\end{reptheorem}

Moreover, we present \ref{example-local-red}, which shows that this additional assumption of local irreducibility cannot be dropped. 

\vspace{5pt}

A key instrument in our work is the blow-up of a complex space with singularities, for which two equivalent definitions prove to be useful in different stages of the proof: the blow-up as the closure of a graph, and the blow-up \textit{via} global analytic $\Proj$. The crucial fact we use repeatedly, that the blow-up of a compact \K \ space is itself \K, is well known (see \cite{var86}, \cite{fuji78}), but as we could not find a self-contained proof in the literature, we also provided one before passing to our main result.

\newpage

The paper is organized as follows:

In Section \ref{sec:preliminaries}, we give all the definitions needed when working with lcK spaces. We also present, using \v Cech cohomology, a simplified proof of \cite[Corollary 2.10]{PS}, the main ingredient of that paper, which we previously proved using elementary properties of integration along curves. We then collect some known results about blow-ups of complex spaces, which can be found in \cite{peternell}.

In Section \ref{sec:blowups-kahler}, we adapt Fujiki's ideas \cite[Lemma 2]{fuji75} to give \ref{blow-up-Kahler}, a self-contained proof of the fact that blow-ups of K\" ahler spaces are also K\" ahler.

Finally, in Section \ref{sec:vaisman-singular}, we prove \ref{vaisman-singular}, which is the main result of this article. Then, we give \ref{example-local-red} which shows that the additional condition of local irreducibility is essential.

\section{Preliminaries}\label{sec:preliminaries}

We begin this section by giving the definitions of the main objects we will be studying, which we introduced in this form in \cite{PS}. 

\subsection{Locally conformally K\" ahler spaces}
For convenience, throughout this article, \textit{psh} stands for \textit{plurisubharmonic}. 

\begin{definition}
	Consider $X$ to be a complex space and $\varphi:X\to\mathbb{R}$. Then, $\varphi$ is said to be smooth if for every $x\in X$, there exists an open neighborhood $U\ni x$ which can be embedded as an analytic subset of an open subset $D$ of some Euclidean space, $U\subset D\hookrightarrow \mathbb{C}^N$, and a smooth function $\tilde{\varphi}:D\to \mathbb{R}$ such that $\tilde{\varphi}_{|U}=\varphi_{|U}$. We have the same definition for \textit{psh} instead of \textit{smooth}. 
\end{definition}

\begin{remark}
	There may exist a smooth and psh function $\varphi:X\to\mathbb{R}$ which cannot be locally realized as the restriction of a function $\tilde{\varphi}$ on a local embedding of $X$ which is at the same time smooth and psh.
\end{remark}

\vspace{5pt}

For the first definition in this subsection, we only require $X$ be a topological space.

\begin{definition}
	\label{def:1-form}
	\begin{enumerate}[1)]
		\item Denote $
				\widetilde{Z}^1(X) = \{ (U_\alpha, f_\alpha)_{\alpha \in A}\}$, where  $(U_\alpha)_{\alpha\in A}$ is an open covering of $X$, $f_\alpha : U_\alpha \to \mathbb{C}$ continuous, and $f_\alpha - f_\beta$ locally constant on $U_\alpha \cap U_\beta, \forall \alpha, \beta \in A$.	
		We define an equivalence relation on $\widetilde{Z}^1(X)$ by 
		\[
		(U_\alpha, f_\alpha)_{\alpha \in A} \sim (V_\beta, h_\beta)_{\beta \in B} \iff (U_\alpha, f_\alpha)_{\alpha \in A} \cup (V_\beta, h_\beta)_{\beta \in B} \in \widetilde{Z}^1(X).
		\]
		
		\item We define the space of \textit{topologically closed 1-forms} on $X$ to be the quotient space $Z^1(X) = \widetilde{Z}^1(X)/\sim$. An element $\eta \in Z^1(X)$ is called a \textit{topologically closed 1-form}. 
		
		\item An element $\theta \in Z^1(X)$ is called \textit{exact} if $\eta = \widehat{(X, f)}$ for a continuous $f:X \to \mathbb{C}$. In this case, we make the notation $\eta = df$.
	\end{enumerate}
\end{definition}

\begin{remark}
	\label{rem:1-formSing}
	Closed 1-forms on smooth manifolds are topologically closed 1-forms.
\end{remark}

\begin{remark}
	\label{rem:1-formCech}
	The set of topologically closed $1$-forms can be equivalently described in terms of \v Cech cohomology:
	\[
	Z^1(X) = \check{\mathrm{H}}^0\left(X,\faktor{\mathscr{C}}{\underline{\R}}\right),
	\]
	where $\mathscr{C}$ is the sheaf of continuous functions on $X$ and $\underline{\R}$ is the locally constant sheaf.
\end{remark}

\vspace{3pt}

The last remark allows for a much shorter proof of a critical result in our previous paper, which we include below:

\begin{theorem}\textnormal{(\cite[Corollary 2.10]{PS})}
	If $X$ is simply connected, then every topologically closed $1$-form is exact.
	
\end{theorem}

\begin{proof}
	Consider the short exact sequence
	\[
	0 \to \underline{\R} \to \mathscr{C} \to \faktor{\mathscr{C}}{\underline{\R}} \to 0
	\]
	
	We then pass to the long exact sequence in cohomology:
	\[
	0 \to H^0(X, \underline{\R}) \to H^0(X, \mathscr{C}) \to H^0(X, \faktor{\mathscr{C}}{\underline{\R}}) \to H^1(X, \underline{\R}) \to ...
	\]
	
	Let $\eta$ be a topologically closed $1$-form. By \ref{rem:1-formCech}, $\eta \in \check{\mathrm{H}}^0\left(X,\faktor{\mathscr{C}}{\underline{\R}}\right)$. Moreover, as $X$ is simply connected, $H^1(X, \underline{\R}) = 0$. We then obtain immediately that $\eta$ is exact.
\end{proof}

\vspace{5pt}

Let $X$ be a connected complex space of dimension $n$.  The definition of \textit{K\" ahler space} we use is the one given by Grauert \cite{grauert}:

\begin{definition}\label{def:Kahler}
	\begin{enumerate}[1)]
		\item Denote 
				$\widetilde{\mathcal{K}}(X) = \{  (U_\alpha, \varphi_\alpha)_{\alpha \in A} \}$, where $(U_\alpha)_\alpha$ is an open cover of $X$,  $\varphi_\alpha : U_\alpha \to \mathbb{R}$ is smooth and strongly psh, and $\ei \del \delb (\varphi_\alpha - \varphi_\beta) = 0$ on $U_\alpha \cap U_\beta \cap X_{\reg}$,  for any $\alpha,\beta\in A$.
		We define an equivalence relation on $\widetilde{\mathcal{K}}(X)$ by 
		\[
		(U_\alpha, \varphi_\alpha)_{\alpha \in A} \sim (V_\beta, \psi_\beta)_{\beta \in B} \iff (U_\alpha, \varphi_\alpha)_{\alpha \in A} \cup (V_\beta, \psi_\beta)_{\beta \in B} \in \widetilde{\mathcal{K}}(X).
		\]
		
		\item We define the space of \textit{K\"ahler forms} on $X$ to be the quotient space $\mathcal{K}(X) = \widetilde{\mathcal{K}}(X)/\sim$. An element $\omega \in \mathcal{K}(X)$ is called a \textit{K\"ahler form} (or \textit{K\"ahler metric}).
		
		\item $(X, \omega)$ is called a \textit{K\"ahler space}. 
	\end{enumerate}	
\end{definition}

\vspace{5pt}

Adapting the most convenient definition of lcK manifolds and using the same idea as in the definition of K\" ahler spaces, we have the following definition for \textit{lcK spaces}. For technical reasons imposed by working with blow-ups, on which we shall construct metrics which satisfy the K\" ahler condition, except for the fact that they are degenerate on the exceptional divisor, we also introduce \textit{weakly locally conformally K\" ahler} (wlcK) forms.

\begin{definition}
	\label{def:LCK}
	\begin{enumerate}[1)]
		\item Denote 
		
		\hspace{10pt} $\widetilde{lc\mathcal{K}}(X) = \{  (U_\alpha, \varphi_\alpha, f_\alpha)_{\alpha \in A} \}$, where $(U_\alpha)_\alpha$ is an open cover of $X$, $\varphi_\alpha : U_\alpha \to \mathbb{R}$ is smooth and strongly psh, $f_\alpha: U_\alpha \to \mathbb{R}$ is smooth, and 
		$$e^{f_\alpha} \ei \del \delb \varphi_\alpha = e^{f_\beta} \ei \del \delb \varphi_\beta$$ 
		on $U_\alpha \cap U_\beta \cap X_{\reg}$,  for any $\alpha,\beta\in A$.

\vspace{5pt}

			\hspace{10pt} $\widetilde{wlc\mathcal{K}}(X) = \{  (U_\alpha, \varphi_\alpha, f_\alpha)_{\alpha \in A}\}$, where $(U_\alpha)_\alpha$ is an open cover of $X$, $\varphi_\alpha : U_\alpha \to \mathbb{R}$ is smooth and psh, and strongly psh outside an analytic set of positive codimension, and $f_\alpha: U_\alpha \to \mathbb{R}$ is smooth and 
			$$e^{f_\alpha} \ei \del \delb \varphi_\alpha = e^{f_\beta} \ei \del \delb \varphi_\beta$$
			on $U_\alpha \cap U_\beta \cap X_{\reg}$, for any $\alpha,\beta\in A$.

\vspace{5pt}

		We define equivalence relations on $\widetilde{lc\mathcal{K}}(X)$ and $\widetilde{wlc\mathcal{K}}(X)$ by 
		\begin{equation*}
			\begin{split}
				(U_\alpha, \varphi_\alpha, f_\alpha)_{\alpha \in A} \sim (V_\beta, \psi_\beta, h_\beta)_{\beta \in B} \iff \\ (U_\alpha, \varphi_\alpha, f_\alpha)_{\alpha \in A} \cup (V_\beta, \psi_\beta, h_\beta)_{\beta \in B} \in \widetilde{lc\mathcal{K}}(X).
			\end{split}
		\end{equation*}
		and similarly,
		\begin{equation*}
			\begin{split}
				(U_\alpha, \varphi_\alpha, f_\alpha)_{\alpha \in A} \sim (V_\beta, \psi_\beta, h_\beta)_{\beta \in B} \iff \\ (U_\alpha, \varphi_\alpha, f_\alpha)_{\alpha \in A} \cup (V_\beta, \psi_\beta, h_\beta)_{\beta \in B} \in \widetilde{wlc\mathcal{K}}(X).
			\end{split}
		\end{equation*}
		
		\item We define the space of \textit{locally conformally K\"ahler (lcK) forms} (or \textit{metrics}) on $X$ to be the quotient space $lc\mathcal{K}(X) =  \widetilde{lc\mathcal{K}}(X)/\sim$ and the space of \textit{weakly locally conformally K\"ahler (wlcK) forms} (or \textit{metrics}) on $X$ to be the quotient space $wlc\mathcal{K}(X) = \widetilde{wlc\mathcal{K}}(X)/\sim$. An element $\omega \in (w)lc\mathcal{K}(X)$ is called an (w)lcK metric. Obviously, $lc\mathcal{K}(X) \subset wlc\mathcal{K}(X)$.
		
		\item The covering $(U_\alpha)_\alpha$ together with the functions $f_\alpha$ give rise to a topologically closed $1$-form $\theta$ on $X$. We call $\theta$ the \textit{Lee form} associated to the (w)lcK metric.
		
		\item If the Lee form $\theta$ is exact, the metric $\omega\in (w)lc\mathcal{K}(X)$ is called \textit{(weakly) globally conformally K\" ahler} (\textit{(w)gcK} for short).
		
		\item $(X, \omega, \theta)$ is called a \textit{(w)lcK space}. If $\theta$ is exact, $(X, \omega, \theta)$ is called a \textit{(w)gcK space}.
	\end{enumerate}	
\end{definition}

\vspace{5pt}

\subsection{Blow-ups of complex spaces}
\label{ssec:blup}

In the next paragraphs we describe two equivalent methods to construct the blow-up of a complex space along a closed subspace, which can be found in \cite{peternell}. In the results that follow, we shall make use of both of them alternatively, according to which is the most practical.

\begin{definition} A subspace $Z$ of a complex space $X$ is called a \textit{Cartier divisor} in $X$ if for every $a\in Z$, $Z$ can be defined locally near $a$ by a single equation $h=0$, where $h\in\mathcal{O}_{X,a}$ is a non-zero divisor.
\end{definition}

\begin{remark}
	A Cartier divisor is a hypersurface in $X$, but not conversely, and its complement $X \setminus Z$ is dense in $X$.
\end{remark}
\vspace{5pt}

The first of the two equivalent definitions of blow-up is the more analytic one, given \textit{via} closure of a graph.

\begin{definition}\label{blow-up graph}(Blow-up as closure of graph) 
	Let $X$ be a complex analytic space and $Z=V(I)$ a closed subspace of $X$ defined by an ideal $I$ of $\mathcal{O}(X)$, generated by elements $g_0,g_1\ldots,g_k$. The morphism 
	\[
	\gamma:X\setminus Z \to \mathbb{P}^{k}, a\mapsto [g_0(a):g_1(a):\ldots:g_k(a)]
	\]
	is well defined. The closure $\widetilde{X}$ of the graph $\Gamma$ of $\gamma$ inside $X\times \mathbb{P}^{k}$ together with the restriction $\pi:\widetilde{X}\to X$ of the projection $X\times\mathbb{P}^{k}\to X$ is the blow-up of $X$ along $Z$. 
	It does not depend, up to isomorphism over $X$, on the choice of the generators $g_i$ of $I$. Hence, by gluing we may construct the blow-up along arbitrary closed subspaces. 
	$\pi^{-1}(Z)$ is a Cartier divisor, in particular a hypersurface, and is called the exceptional divisor or exceptional locus of the blow-up. $Z$ is called the center of the blow-up.
\end{definition}

\vspace{5pt}

The second definition is a more algebraic construction of the blow-up \textit{via} global analytic $\Proj$, denoted $\Proj_{\an}$. 

\vspace{5pt}

\begin{definition}\label{Blow-up as Proj}(Blow-up as $\Proj_{\an}$) 
	Let $X$ be a complex space and $Z\subset X$ a closed subspace. Let $\mathscr{I}\subset \mathcal{O}_X$ be the ideal sheaf of $Z$ in $X$ and denote by $\mathscr{A}=\bigoplus_{m=0}^{\infty}\mathscr{I}^m$ the graded $\mathcal{O}_X$-algebra, such that for an open subset $U\subset X$, the elements of $\mathscr{A}_m(U)=\mathscr{I}(U)^m$ have degree $m$. Since $\mathscr{I}$ is coherent, it can be shown that $\mathscr{A}$ is of finite presentation, \textit{i.e.} locally, we have the isomorphism $\mathscr{A}=\mathcal{O}_U[t_0,t_1,\ldots,t_n]/\mathscr{J}$, with a finitely generated homogenous ideal $\mathscr{J}$ in $\mathcal{O}_U[t_0,t_1,\ldots,t_n]$.
	
	Let $\Proj_{\an}(\mathscr{A}_{|U})\subset U\times \mathbb{P}^n$ be the subspace defined by $\mathscr{J}$ and consider the projection on the first component $\pi_{|U}:\Proj_{\an}(\mathscr{A}_{|U}))\subset U\times \mathbb{P}^n\to U$. The blow-up of $X$ with center $Z$ is the space $\widetilde{X}=\Proj_{\an}(\mathscr{A})$, obtained by glueing together these local pieces, together with the projection $\pi:\widetilde{X}\to X$.
	
	With this construction, there exists a canonical line bundle (which can be seen as an invertible sheaf) $\mathcal{O}(1)$ on $\widetilde{X}$, such that on $\pi^{-1}(U)\subset U\times\mathbb{P}^n$, we have $\mathcal{O}(1)=i^*p_2^*(\mathcal{O}_{\mathbb{P}^n}(1))$, where $i:\pi^{-1}(U)\to U\times\mathbb{P}^n$ is the inclusion, $p_2$ is the projection on the second factor in $U\times \mathbb{P}^n$, and $\mathcal{O}_{\mathbb{P}^n}(1)$ is the canonical line bundle of $\mathbb{P}^n$. 
\end{definition}

\section{Blow-ups of K\"ahler spaces}\label{sec:blowups-kahler}

In this section, we give a self-contained proof that the blow-up of a compact K\" ahler space along a closed subspace admits a K\" ahler metric. The main ideas of this proof come from \cite[Lemma 2]{fuji75}, and for most elements of the proof, we keep Fujiki's notations.

\begin{theorem}\label{blow-up-Kahler}
	Let $X$ be a compact K\"aher space and $Z\subset X$ a closed subspace of positive codimension. Then, the blow-up $\pi:\widetilde{X}\to X$ of $X$ along $Z$ admits a K\"ahler metric.
\end{theorem}
\begin{proof}
	For this proof, we use the algebraic \ref{Blow-up as Proj} of blow-up.
	
	Let $\mathscr{A}_X$ (resp. $\mathscr{A}_{\widetilde{X}}$) be the sheaf of differentiable functions on $X$ (resp. $\widetilde{X}$); then $\mathscr{A}_X\otimes_{\mathcal{O}_X}\pi_*(\mathcal{O}(1))$ is a sheaf of $\mathscr{A}_X$-modules. Since $\mathcal{O}(1)$ is a coherent sheaf of $\mathcal{O}_{\widetilde{X}}$-modules and $\pi$ is a proper mapping, $\pi_*(\mathcal{O}(1))$ is also coherent, hence locally finitely generated. Consider an open covering $\mathscr{U}=\{U_\alpha\}_{\alpha \in A}$ of $X$ such that for every $\alpha\in A$, there are finitely many sections $s_{\alpha,j}\in \pi_*(\mathcal{O}(1))(U_\alpha)$, $j=1,\ldots,n_\alpha$, which generate $\pi_*(\mathcal{O}(1))_x$ for any $x\in U_\alpha$. Since $X$ is compact, we may assume that the covering $\{U_\alpha\}_{\alpha \in A}$ is finite. $X$ is paracompact, thus there exists a partition of unity $\{\chi_\alpha\}_{\alpha \in A}$ subordinate to the covering $\{U_\alpha\}_{\alpha \in A}$. Then, 
	$$\Phi_0:=\{\chi_\alpha\otimes s_{\alpha,j}\}_{\alpha,j}\subset H^0(X,\mathscr{A}_X\otimes_{\mathcal{O}_X}\pi_*(\mathcal{O}(1)))$$
	is a finite set of sections which generate the stalk of $\mathscr{A}_X\otimes_{\mathcal{O}_X}\pi_*(\mathcal{O}(1))$ at every point of $X$. By taking the pull-back of these sections to $\widetilde{X}$, we obtain a finite set $\Phi\subset H^0(\widetilde{X}, \mathscr{A}_{\widetilde{X}}\otimes_{\mathcal{O}_{\widetilde{X}}}\mathcal{O}(1))$ of sections which are holomorphic along each fiber of $\pi$ and which have no base points on $\widetilde{X}$ (\textit{i.e.} for any $y\in \widetilde{X}$, there exists $\varphi\in\Phi$ such that $\varphi(y)\not =0$). To simplify the notations, we designate the sections of $\Phi$ by $\{\varphi_1,\varphi_2,\ldots,\varphi_r\}$.
	
	Assume that $\mathcal{O}(1)$, as a line bundle, is defined by the transition functions $\{f_{\sigma\mu}\}$ with respect to the covering $\mathscr{V}=\{V_\sigma\}_{\sigma\in S}$ of $\widetilde{X}$. Every section $\varphi_j$ can be expressed as a family of differentiable functions $\{\varphi_{j,\sigma}:V_\sigma\to\mathbb{R}\}$ such that $\varphi_{j,\sigma}=f_{\sigma\mu}\varphi_{j,\mu}$ on $V_\sigma\cap V_{\mu}$.
	Consider the functions $h_\sigma:V_\sigma\to\mathbb{R}$ defined by
	\[
	h_\sigma= \frac{1}{\sum_{j=0}^{r}|\varphi_{j,\sigma}|^2}.
	\]
	$h_\sigma$ are positive smooth functions which satisfy $\frac{h_\mu}{h_\sigma}=|f_{\sigma\mu}|^2$ on $V_\sigma\cap V_{\mu}$, hence giving a metric on the line bundle $\mathcal{O}(1)$ on $\widetilde{X}$. 
	
	Now, we may assume that the covering $\mathscr{U}=\{U_\alpha\}_{\alpha \in A}$ of $X$ was chosen such that for every $\alpha$, $\pi^{-1}(U_\alpha)$ is embedded as a closed subspace of $U_\alpha\times \mathbb{P}^n$ for an integer $n$ which depends on $\alpha$. We may also assume that the covering $\mathscr{V}=\{V_\sigma\}_{\sigma\in S}$ of $\widetilde{X}$ was chosen such that each $V_\sigma$ is embedded as a closed subspace of $U_\alpha^\prime\times G_\sigma$ for some $\alpha$, where $U_\alpha^\prime$ is a subdomain of $U_\alpha$, and $G_\sigma$ is a subdomain of $\mathbb{P}^n$ which has the coordinates $(w_i^\sigma)$. Then, $h_\sigma$ is the restriction of a $\mathcal{C}^\infty$-smooth function $\tilde{h}_\sigma:U_\alpha^\prime\times G_\sigma\to\mathbb{R}$, and the Hessian 
	\[
	-\frac{\partial^2 \log \tilde{h}_\sigma}{\partial w_i^\sigma \partial \overline{w_j^\sigma}}
	\]
	of $-\log \tilde{h}_\sigma$ with respect to the coordinates $(w_i^\sigma)$ is positive definite at every point of $U_\alpha^\prime\times G_\sigma$. Hence, the restriction to every fiber $\pi^{-1}(x)$ of $-\log h_\sigma$ is strictly psh. 
	
	Moreover, since $\frac{h_\mu}{h_\sigma}=|f_{\sigma\mu}|^2$ on $V_\sigma\cap V_{\mu}$ and $f_{\sigma\mu}$ is holomorphic,
	\[
	\ei\partial \overline{\partial} (-\log h_\sigma)=\ei\partial \overline{\partial} (-\log h_\mu) \text{\ on\ }V_\sigma\cap V_\mu\cap {\widetilde{X}}_{\reg}.
	\] 
	Finally, consider $\omega=(W_\eta,\psi_\eta)_{\eta\in H}$ a K\" ahler metric on $X$ and take its pull-back $\pi^*\omega=(\pi^{-1}(W_\eta), \pi^*\psi_\eta)_{\eta\in H}$. If $c>0$ is a sufficiently small constant, then
	\[
	\Omega=(\pi^{-1}(W_\eta)\cap V_\sigma, \pi^*\psi_\eta-c\log h_\sigma)_{\eta\in H \atop \sigma \in S}
	\]
	is a K\" ahler metric on $\widetilde{X}$.
\end{proof}

\section{Vaisman theorem for lcK spaces}\label{sec:vaisman-singular}

We now come to the main results of this article. First, we prove some facts about blow-ups of locally irreducible complex spaces. Then, we show that if the pull-back of a wlcK metric on the blow-up of a locally irreducible complex space is wgcK, then the initial metric is also wgcK. Next, we prove that Vaisman's theorem on the \textit{pure lcK} vs \textit{K\" ahler} dichotomy still holds for locally irreducible lcK spaces. In the end, we show by an example that the assumption of local irreducibility is essential in our generalization.

In this section, we use only \ref{blow-up graph} of the blow-up.

\begin{lemma}\label{lemma-blowup-ired}
	Let $X$ be a locally irreducible complex space and $Z\subset X$ a closed subspace of positive codimension. Then, the blow-up $\pi:\widetilde{X}\to X$ of $X$ along $Z$ is locally irreducible.
\end{lemma}
\begin{proof}We argue by contradiction.
	Assume that $\widetilde{X}$ is reducible at $\widetilde{x}$. Then, if $\widetilde{U}\subset\widetilde{X}$ is a sufficiently small, open, connected neighborhood of $\widetilde{x}$, there exists a proper decomposition in open subsets $\widetilde{U}=\widetilde{U}_1\cup \widetilde{U}_2$ such that $\widetilde{A}=\widetilde{U}_1\cap \widetilde{U}_2$ is an analytic set of codimension at least $1$ in $\widetilde{U}$. Denote $U=\pi(\widetilde{U})$, which is also open and connected. The local irreducibility of $X$ implies that $U\setminus Z$ is still connected.
	Since $\widetilde{X}\setminus \pi^{-1}(Z)\simeq X\setminus Z$ and $X$ is locally irreducible, we get $\widetilde{A}\subset\pi^{-1}(Z)$ which implies that $\widetilde{U}\setminus \pi^{-1}(Z)$ is disconnected. However, $\widetilde{U}\setminus \pi^{-1}(Z)\simeq U\setminus Z$, and $U\setminus Z$ is connected, which is a contradiction.
\end{proof}

\begin{lemma}\label{lemma-preimagine-conexa}
	Let $X$ be a locally irreducible complex space, $Z\subset X$ a closed subspace of positive codimension, and $\pi:\widetilde{X}\to X$ the blow-up of $X$ along $Z$. If $U\subset X$ is an open, connected set, then $\pi^{-1}(U)$ is connected. 
\end{lemma}
\begin{proof}
	$\pi^{-1}(Z)$ is a Cartier divisor in $\widetilde{X}$, so $\codim_{\widetilde{X}}\pi^{-1}(Z)=1$. Given that $X$ is locally irreducible and $U$ is connected, we have $U\setminus Z$ also connected. Moreover, $U\setminus Z \simeq \pi^{-1}(U\setminus Z)=\pi^{-1}(U)\setminus\pi^{-1}(Z)$. Since $$\pi^{-1}(U)\setminus\pi^{-1}(Z)\subset\pi^{-1}(U)\subset\overline{\pi^{-1}(U)\setminus\pi^{-1}(Z)},$$ we obtain that $\pi^{-1}(U)$ is connected.
\end{proof}

\begin{theorem}\label{blow-up-lck}
	Let $(X,\omega,\theta)$ be a wlcK, locally irreducible complex space. Consider $Z\subset X$ a closed complex subspace of positive codimension and $\pi:\widetilde{X}\to X$ the blow-up of $X$ along $Z$. If $\pi^* \omega$ is wgcK, then $\omega$ is wgcK.
\end{theorem}
\begin{proof}
	If $\pi^* \omega$ is wgcK, then there exists a smooth function $\widetilde{f}:\widetilde{X}\to\mathbb{R}$ such that $d\widetilde{f}=\pi^* \theta$. We want to prove that $\widetilde{f}$ descends to a smooth function $f:X\to\mathbb{R}$, so we have to show that $\widetilde{f}$ is constant on the fibers of $\pi$. Take $x\in X$. If $x\not \in Z$, then $\pi^{-1}(x)$ consists of a single point, so we only have to study the case $x\in Z$. For $x\in Z$, consider $U$ an open, connected neighborhood of $x$ in $X$ such that $(\omega,\theta)$ is represented by $(\varphi, h)$ on $U$. According to \ref{lemma-preimagine-conexa}, $\pi^{-1}(U)$ is connected. We have $d\widetilde{f}=\pi^*\theta=\pi^* dh=d\pi^*h$, thus $\widetilde{f}=\pi^*h+C$ on $\pi^{-1}(U)$, where $C$ is a real constant. Consequently, $\widetilde{f}$ is constant on each fiber of $\pi$ above $U$. Thus, $\widetilde{f}$ descends to a function $f:X\to\mathbb{R}$, which is smooth, since, if $\theta=\reallywidehat{(U_\alpha,h_\alpha)_\alpha}$, then $f-h_\alpha$ is locally constant on $U_\alpha$ for each $\alpha$. Finally, $e^{-f}\omega$ is weakly K\" ahler, which means that $\omega$ is wgcK.
\end{proof}

\vspace{5pt}

Now, using \ref{blow-up-Kahler} and \ref{blow-up-lck}, we can prove the following generalization for Vaisman's theorem on lcK compact manifolds.

\begin{theorem}\label{vaisman-singular}
	Let $(X,\omega,\theta)$ be a compact, locally irreducible, lcK space. If $X$ admits a K\"ahler metric, then $(X,\omega,\theta)$ is gcK.
\end{theorem}
\begin{proof}
	Since $X$ is a compact complex space, a resolution of singularities for $X$ can be constructed by a finite sequence of blow-ups $$\pi_j:X_j\to X_{j-1}, j=1,\ldots,m,$$ such that $X_m=:\widehat{X}$ is a compact manifold, and $X_0=X$. A repeated application of \ref{blow-up-Kahler} yields that $\widehat{X}$ admits a K\" ahler metric. Denote $\pi=\pi_1\circ\pi_2\circ\cdots\circ\pi_m:\widehat{X}\to X$. Then, $(\widehat{X},\pi^*\omega,\pi^*\theta)$ is a wlcK manifold which admits a K\" ahler metric. By \cite[Lemma 2.5]{APV}, $(\widehat{X},\pi^*\omega,\pi^*\theta)$ is wgcK.
	Also, \ref{lemma-blowup-ired} ensures that $X_j$ is locally irreducible for every $j=1,\ldots,m$. 
	 Now, using \ref{blow-up-lck}, after $m$ steps, we obtain that $\omega$ is wgcK. Since $\omega$ is not degenerate, because it is assumed to be lcK, we deduce that it is gcK. 
\end{proof}

\begin{example}\label{example-local-red}
	The assumption of local irreducibility in \ref{vaisman-singular} cannot be dropped. For this, we present the following counterexample, constructed by Vuletescu:
	
	We take two distinct points $v,w\in\mathbb{P}^n$ and consider the space obtained by identifiying these two points, $X=\mathbb{P}^{n}/(v\simeq w)$. $X$ has one singular point $[v]$, and the germ $(X,[v])$ is biholomorphic to $$\left(\{{(z,t)\in\mathbb{C}^n\times\mathbb{C}^n | z=0 \text{ or } t =0}\},0\right),$$ so $X$ is locally reducible. By \ref{identifying-2-points} proved below,
	$\pi_1(X)=\pi_1(\mathbb{P}^n)*\mathbb{Z}=\mathbb{Z}$. We now give a description for the universal cover $\widetilde{X}$ of $X$. Denote $\widetilde{X}_0=\mathbb{P}^n\times \mathbb{Z}$ and define an equivalence relation on $\widetilde{X}_0$ by setting $(w,k)\sim(v,k+1)$ for all $k\in\mathbb{Z}$. Then, put $\widetilde{X}=\widetilde{X}_0/\sim$. The deck transform $\gamma:\widetilde{X}\to\widetilde{X}$, $\gamma([z,k])=[z,k+1]$ generates $\pi(X)$, where $[z,k]$ denotes the class of equivalence of $(z,k)\in\mathbb{P}^n\times\mathbb{Z}$. 
	
	Note that each globally irreducible component of $\widetilde{X}$ is a copy of $\mathbb{P}^n$ and comes from a connected component $\widetilde{X}_{0,k}:=\mathbb{P}^n\times\{k\}$ of $\widetilde{X}_0$, so we will denote it $\widetilde{X}_k$.
	Now, we construct the following K\" ahler metrics on $\widetilde{X}$:
	\begin{itemize}
		\item $\widetilde{\omega}_1$, which is the Fubini-Study metric $\omega_{FS}$ on $\widetilde{X}_k$;
		\item $\widetilde{\omega}_2$, which is defined by $\left({\widetilde{\omega}_2}\right)_{|\widetilde{X}_k}:=2^k\omega_{FS}$, for every $k\in\mathbb{Z}$.
	\end{itemize} 
	 Finally, $\widetilde{\omega}_1$ defined this way is a K\" ahler metric on $\widetilde{X}$ which is invariant under $\gamma$, hence it descends to a K\" ahler metric on $X$. However, according to the proof of \cite[Theorem 3.10]{PS}, $\widetilde{\omega}_2$ induces a lcK (non-gcK) metric on $X$. 
	
	\begin{lemma}\label{identifying-2-points}
		Let $X$ be a path-connected and locally contractible topological space and $x,y \in X$ two distinct points. Denote by $\widehat{X}$ the topological space obtained by identifying $x$ and $y$ in $X$. Then we have $\pi_1(\widehat{X})=\pi_1(X)*\mathbb{Z}$.
	\end{lemma}
	\begin{proof}
		It is easy to see that $\widehat{X}$ is homotopically equivalent to $X$ with a glued segment between $x$ and $y$, \ie 
		\[
		\faktor{\left(X \coprod \ [0, 1]\right)}{\sim}, \text{ where } x \sim 0 \text{ and } y \sim 1.
		\]
		Since $x$ and $y$ can be connected in $X$ along a path and $X$ is locally contractible, this is further homotopically equivalent to $X \vee S^1$. A straightforward application of Van Kampen's theorem gives the desired conclusion.
	\end{proof}
	 
\end{example}

\vspace{5pt}

\textbf{Acknowledgments.} We would like to thank Liviu Ornea for our useful discutions while working on this article and Victor Vuletescu for constructing the example which provided the key assumption we were missing in our theorem.

\end{document}